\documentclass [a4paper, 12pt]{amsart}

\usepackage{times}

\usepackage{amsmath,amsfonts,amsthm,amssymb,amscd}
\usepackage{graphicx}
\usepackage{epstopdf}
\usepackage{fancyhdr}
\usepackage[all]{xy}

\newtheorem{cor}{Corollary}[section]
\newtheorem{te}[cor]{Theorem}
\newtheorem{p}[cor]{Proposition}

\theoremstyle{definition}

\theoremstyle{remark}

\newcommand{\nz}{\mathbb{N}}

\newcommand{\vp}{\varepsilon}

   {\begin{verse}%
     \small }%
   {\end{verse}}

\begin{document}

\title{All automorphisms of the universal sofic group are class-preserving}
\maketitle
\begin{center}
Liviu P\u aunescu\footnote{Work supported by a grant of the Romanian National Authority
for Scientific Research, CNCS - UEFISCDI, project number
PN-II-ID-PCE-2012-4-0201‏}
\end{center}

\begin{center}
\emph{Dedicated to Professor \c{S}erban Str\u{a}til\u{a} on his 70th birthday.}
\end{center}

 $\mathbf{Abstract.}$ In this article we prove that every automorphism of a universal sofic group is class preserving.

\tableofcontents

\section{Introduction}

The universal sofic group was introduced by Elek and Szabo in \cite{El-Sza} in order to provide a new characterization of sofic groups. In this setting a countable group is sofic if and only if there exists an injective morphism into the universal sofic group. A \emph{sofic representation} of a countable group is a morphism from the group into the universal sofic group that is faithful in a certain sense.

In \cite{El-Sza2} the same authors proved that any two sofic representations of a group are conjugate by an inner automorphism of the universal sofic group if and only if the group is amenable. The space of sofic representations factored by the equivalence relation given by conjugacy using inner automorphisms was introduced in \cite{Pa} following ideas from \cite{Br}. A combination of arguments from \cite{El-Sza2} and the appendix of \cite{Br} written by Taka Ozawa shows that this space is non-separable for a sofic non-amenable group. So, it is a natural question to see if there is a different outcome when one uses general automorphisms instead of just inner automorphisms of the universal sofic group. Here we show that each such automorphism is class-preserving, i.e. preserves the conjugacy classes. Recenlty Capraro and Lupini showed that, if we assume the Continuum Hypothesis, then there exists non-inner automorphisms (Corollary II.4.7 of \cite{Ca-Lu}). The problem is still open in the absence of CH.

\subsection{Notation}

For $n\in\nz^*$ we denote by $S_n$ the symmetric group on the set $\{1,2,\ldots,n\}$. The normalized Hamming distance on this group is defined as:
\[d_{Hamm}(p,q)=\frac1n|\{a:p(a)\neq q(a)\}|,\]
i.e. it counts the number of points where $p$ and $q$ are different as functions.

For $i\in\nz^*$ denote by $cyc_i(p)$ the number of cycles of size $i$ in $dcd(p)$ multiplied by $i$, so that $n=\sum_{i\geqslant 1}cyc_i(p)$. Then $cyc_1(p)$ denotes the number of fixed points and it is easy to see that $d_{Hamm}(p,Id)=1-\frac1ncyc_1(p)$.

Following the notation in \cite{HKL} we denote by $m(p)$ the size of the support of $p$ and by $n(p)$ the number of nontrivial cycles in $p$ for $p\in S_n$. Then:
\[m(p)=n-cyc_1(p);\ \ \ n(p)=\sum_{i\geqslant 2}\frac{cyc_i(p)}{i}.\]

\subsection{The universal sofic group}

From now on let $\omega$ be a non-principal (free) ultrafilter and let $\{n_k\}_k\subset\nz^*$ be a sequence such that $\lim_kn_k=\infty$. Let $\Pi_kS_{n_k}$ be the Cartesian product and define: \[\mathcal{N}_\omega=\{(p_k)_k\in\Pi_kS_{n_k}:\lim_{k\to\omega} cyc_1(p_k)/n_k=1\}=\{(p_k)_k:\lim_{k\to\omega} d_{Hamm}(p_k,Id)=0\}.\] Due to the bi-invariance of the Hamming distance, $\mathcal{N}_\omega$ is a normal subgroup of $\Pi_kS_{n_k}$ so we can define the universal sofic group as: $\Pi_{k\to\omega}S_{n_k}=\Pi_kS_{n_k}/\mathcal{N}_\omega$.

Denote by $\Pi_{k\to\omega}p_k$ the generic element of the universal sofic group, where $p_k\in S_{n_k}$. By construction the Hamming distance can be extended on this ultraproduct, so for an element  $p=\Pi_{k\to\omega}p_k\in\Pi_{k\to\omega}S_{n_k}$ the value $cyc_1(p)=\lim_{k\to\omega}\frac{cyc_1(p_k)}{n_k}$ is well defined, i.e. it does not depend on the particular choice of elements $p_k\in S_{n_k}$ as long as $p=\Pi_{k\to\omega}p_k$. Actually, for each $i\in\nz^*$ we define:
\[cyc_i(p)=\lim_{k\to\omega}\frac{cyc_i(p_k)}{n_k}.\]
It came as a surprise to me that these numbers are well defined when I first read about this in \cite{El-Sza}. This result can be deduced using the numbers $cyc_1(p)$. For example $cyc_2(p)=cyc_1(p^2)-cyc_1(p)$. In general $cyc_1(p^i)=\sum_{j|i}cyc_j(p)$ and using an inclusion-exclusion principle one can reach the formula (we shall not use it):
\[cyc_i(p)=\sum_{(\vp_1,\ldots,\vp_t)\in\{0,1\}^t}(-1)^{\vp_1+\ldots+\vp_t}cyc_1(p^{a_1^{r_1-\vp_1}\ldots a_t^{r_t-\vp_t}}),\]
where $i=a_1^{r_1}\ldots a_t^{r_t}$ is the decomposition into prime numbers of $i$.

It is an easy exercise to see that $\sum_{i\geqslant 1}cyc_i(p)\leqslant 1$. Let $cyc_\infty(p)=1-\sum_{i\geqslant 1}cyc_i(p)$.
Lastly, define $m(p)=\lim_{k\to\omega}\frac{m(p_k)}{n_k}$ and $n(p)=\lim_{k\to\omega}\frac{n(p_k)}{n_k}$. We can see that:
\[m(p)=1-cyc_1(p);\ \ \ n(p)=\sum_{i\geqslant 2}\frac{cyc_i(p)}{i}.\]

\subsection{Characterizing conjugacy classes} 

It is well known that two permutations $p,q\in S_n$ are conjugate if and only if $cyc_i(p)=cyc_i(q)$ for any $i\in\nz^*$. Due to a theorem of Elek and Szabo the same is true for elements in the ultraproduct $\Pi_{k\to\omega}S_{n_k}$. We reproduce the proof here.

\begin{p}[Proposition 2.3(4),\cite{El-Sza}]\label{conjugacyclasses}
Two elements $p,q\in\Pi_{k\to\omega}S_{n_k}$ are conjugate if and only if $cyc_i(p)=cyc_i(q)$ for all $i\in\nz^*$.
\end{p}
\begin{proof}
If $p$ and $q$ are conjugate then we can find $p_k,q_k$ such that $p=\Pi_{k\to\omega}p_k$, $q=\Pi_{k\to\omega}q_k$ and each $p_k$ is conjugate to $q_k$. The result will then follow.

For the converse assume that $cyc_i(p)=cyc_i(q)$ for any $i$. Let  $p=\Pi_{k\to\omega}p_k$ and $q=\Pi_{k\to\omega}q_k$. Define $c(i,k)=\min\{cyc_i(p_k),cyc_i(q_k)\}$. Then:
\[\lim_{k\to\omega}\frac{cyc_i(p_k)}{n_k}=\lim_{k\to\omega}\frac{cyc_i(q_k)}{n_k}=\lim_{k\to\omega}\frac{c(i,k)}{n_k}.\]
For each $k\in\nz$ divide the cycles of $p_k$ into two groups: let $P(k)$ be a collection of cycles of $p_k$ that contains exactly $c(i,k)/i$ cycles of length $i$ and let $E(k)$ be the collection of the remaining cycles. Denote by $e(k)$ the cardinality of $E(k)$.
We want to show that $\lim_{k\to\omega}{e(k)}/{n_k}=0$. Let $e(i,k)$ be the number of cycles of length $i$ in $E(k)$. We know that $\lim_{k\to\omega}e(i,k)/n_k=0$ for any $i\in\nz^*$. Then for any $T\in\nz^*$:
\[\lim_{k\to\omega}\frac {e(k)}{n_k}=\lim_{k\to\omega}\sum_{i<T}\frac {e(i,k)}{n_k}+\lim_{k\to\omega}\sum_{i\geqslant T}\frac {e(i,k)}{n_k}\leqslant\frac 1T.\]
As $T$ was arbitrary we get that $\lim_{k\to\omega}{e(k)}/{n_k}=0$. Then we can glue the cycles in $E(k)$ into one large cycle without changing the value of $p=\Pi_{k\to\omega}p_k$. Perform the same operations with $q_k$. Now $p_k$ and $q_k$ are conjugate for any $k$ so $p$ and $q$ are conjugate.
\end{proof}

\section{Solving ultraproducts of cycles}

The first step in the proof of the main result is to deal with ultraproducts of permutations composed of only one cycle. Define:
\[cyc(\{1,\infty\}):=\{p\in\Pi_{k\to\omega}S_{n_k}:cyc_i(p)=0\ \forall i\ 2\leqslant i<\infty\}.\]
From the definition we can deduce that $p\in cyc(\{1,\infty\})$ if and only if $cyc_1(p)+cyc_\infty(p)=1$. Using Proposition \ref{conjugacyclasses} it follows that the conjugacy class of each of these elements is completely determined by $cyc_\infty(p)$.

We have an automorphism, an object that preserves the group structure, and we want preservation of conjugacy classes. For the proof we need results describing conjugacy classes using only properties of the group structure. A first such result is the next proposition.

\begin{p}
Let $p\in\Pi_{k\to\omega}S_{n_k}$. Then $p\in cyc(\{1,\infty\})$ if and only if $p^m\in Cl(p)$ for any $m\in\nz^*$.
\end{p}
\begin{proof}
We already mentioned the formula: $cyc_1(p^i)=\sum_{j|i}cyc_j(p)$.

If  $p^m\in Cl(p)$ for any $m\in\nz^*$ then by Proposition \ref{conjugacyclasses} we have $cyc_1(p^m)=cyc_1(p)$ for all $m\in\nz^*$. Using the above formula we get: $\sum_{j|m}cyc_j(p)=cyc_1(p)$. This implies that $cyc_m(p)=0$ for all $m\geqslant 2$. So $p\in cyc(\{1,\infty\})$ by definition.

Let now $m\in\nz^*$. If $p\in cyc(\{1,\infty\})$ then  $cyc_1(p^m)=\sum_{j|m}cyc_j(p)=cyc_1(p)$. In general, we have $cyc_\infty(p^m)=cyc_\infty(p)$. As $cyc_1(p)+cyc_\infty(p)=1$ it follows that  $cyc_1(p^m)+cyc_\infty(p^m)=1$. This, in turn, implies that $cyc_i(p^m)=0=cyc_i(p)$ for $2\leqslant i<\infty$. It follows that $p^m\in Cl(p)$.
\end{proof}

\begin{cor}\label{preserving 1infty}
For an automorhpism $\Phi$ of $\Pi_{k\to\omega}S_{n_k}$ we have $\Phi(cyc(\{1,\infty\}))=cyc(\{1,\infty\})$.
\end{cor}
\begin{proof}
It is well known and easily verifiable that an automorphism sends a conjugacy class into a conjugacy class, i.e. $\Phi(Cl(p))=Cl(\Phi(p))$. 
If follows that  $p^m\in Cl(p)$ iff  $\Phi(p)^m\in Cl(\Phi(p))$. By the previous proposition  $p\in cyc(\{1,\infty\})$ iff  $\Phi(p)\in cyc(\{1,\infty\})$.
\end{proof}

The triangle inequality is an important tool in this section.

\begin{p}
For $p,q\in\Pi_{k\to\omega}S_{n_k}$ we have the following inequality:
\[1-cyc_1(pq)\leqslant (1-cyc_1(p))+(1-cyc_1(q)).\]
\end{p}
\begin{proof}
By the triangle inequality of the Hamming distance we get that: $d_{Hamm}(pq,Id)\leqslant d_{Hamm}(pq,q)+d_{Hamm}(q,Id)$. The propostion follows because of the bi-invaraince, i.e. $d_{Hamm}(pq,q)=d_{Hamm}(p,Id)$ and from the formula: $d_{Hamm}(p,Id)=1-cyc_1(p)$.
\end{proof}

For the proof of the next proposition we need the following notation: $c_{s,t}^n\in S_n$ is defined by the cycle $c_{s,t}^n=(s,s+1,\ldots, t)$.

\begin{p}\label{inftycycles}
Let $p,q\in cyc(\{1,\infty\})$ and $m>1$. Then:
\[q\in Cl(p)^m\Leftrightarrow cyc_\infty(q)\leqslant m\cdot cyc_\infty(p).\]
\end{p}
\begin{proof}
Let $p_i\in Cl(p)$ such that  $q=p_1p_2\ldots p_m$. By the triangle inequality $1-cyc_1(p_1p_2\ldots p_m)\leqslant\sum_{i=1}^m(1-cyc_1(p_i))$. Because  $cyc_1(p_i)=cyc_1(p)$ we get that
$1-cyc_1(q)\leqslant  m\cdot(1-cyc_1(p))$. For any $v\in cyc(\{1,\infty\})$ we have $1-cyc_1(v)=cyc_\infty(v)$ and we are done with the direct implication.

For the reverse implication, assume first that $cyc_\infty(p)<cyc_\infty(q)$. We have to arrange $m$ cycles of normalized length $cyc_\infty(p)$ such that they occupy a space of length $cyc_\infty(q)$. Find numbers $j_k,r_k\in\nz$ such that $1<j_k<r_k<n_k$, $r_k<mj_k$ and $\lim_{k\to\omega}j_k/n_k=cyc_\infty(p)$,  $\lim_{k\to\omega}r_k/n_k=cyc_\infty(q)$. For each $k\in\nz$ find numbers $1=a_k^1\leqslant a_k^2\leqslant\ldots\leqslant a_k^m$ such that $a_k^t+j_k\geqslant a_k^{t+1}$ and $a_k^m+j_k=r_k$. Define now $p_t=\Pi_{k\to\omega}c_{a_k^t,a_k^t+j_k}^{n_k}$ for $t=1,\ldots,m$. It follows that $p_i\in Cl(p)$. 

Let $s^k=c_{a_k^1,a_k^1+j_k}^{n_k}\ldots c_{a_k^m,a_k^m+j_k}^{n_k}$ so that $p_1\ldots p_m=\Pi_{k\to\omega}s^k$. Then $s^k(a)=a$ if $a>r_k$ and  $s^k(a)>a$ if $a\leqslant r_k$ with maximum $m$ exceptions. Then $s^k$ has maximum $m$ cycles on $\{1,\ldots,r_k\}$. These $m$ cycles are not enough to add some value to $cyc_i(p_1\ldots p_m)$ for any $i\in\nz^*$. It follows that $cyc_\infty(p_1\ldots p_m)=cyc_\infty(q)$ so $p_1p_2\ldots p_m\in Cl(q)$. This is enough to deduce that $q\in Cl(p)^m$.

We are left with the case $cyc_\infty(p)\geqslant cyc_\infty(q)$. We construct $m$ elements composed of two cycles: one cycle of normalized length $cyc_\infty(q)$ and one of normalized length $cyc_\infty(p)- cyc_\infty(q)$. We want to keep the first cycles and to cancel out the the second ones. Find numbers $j_k,r_k\in\nz$ such that $1<r_k<j_k<n_k$ and $\lim_{k\to\omega}j_k/n_k=cyc_\infty(p)$,  $\lim_{k\to\omega}r_k/n_k=cyc_\infty(q)$. Let  $p_t=\Pi_{k\to\omega}c_{1,r_k}^{n_k}c_{r_k+1,j_k}^{n_k}$ for $t=1,\ldots,m-1$ and $p_m=\Pi_{k\to\omega}c_{1,r_k}^{n_k}(c_{r_k+1,j_k}^{n_k})^{-(m-1)}$. It follows that $p_t\in Cl(p)$ and $p_1\ldots p_m=\Pi_{k\to\omega}(c_{1,r_k}^{n_k})^m\in Cl(q)$. This will finish the proof.
\end{proof}

Define $cyc(\{\infty\})$ the set composed of elements $p\in\Pi_{k\to\omega}S_{n_k}$ such that $cyc_\infty(p)=1$. This is actually a conjugacy class.

\begin{te}
Let $p\in cyc(\{1,\infty\})$ and $m>1$. Then the following are equivalent:
\begin{enumerate}
\item $cyc(\{\infty\})\subset Cl(p)^m$;
\item $cyc_\infty(p)\geqslant\frac 1m$;
\item $cyc(\{1,\infty\})\subset Cl(p)^m$.
\end{enumerate}
\end{te}
\begin{proof} The implication $(1)\Rightarrow(2)$ is a particular case of the previous proposition when $cyc_\infty(q)=1$. Also $(2)\Rightarrow(3)$ is a particular case of the same proposition
when $m\cdot cyc_\infty(p)\geqslant 1$. The implication  $(3)\Rightarrow(1)$ is trivial.
\end{proof}

\begin{cor}
For an automorhpism $\Phi$ of $\Pi_{k\to\omega}S_{n_k}$ and $p\in cyc_\infty(\{1,\infty\})$ we have:
\[ cyc_\infty(p)\in\big[\frac1m,\frac1{m-1}\big)\Leftrightarrow cyc_\infty(\Phi(p))\in\big[\frac1m,\frac1{m-1}\big).\]
\end{cor}
\begin{proof}
From Corollary \ref{preserving 1infty} we have  $\Phi(cyc(\{1,\infty\}))=cyc(\{1,\infty\})$. It follows that:
\[cyc(\{1,\infty\})\subset Cl(p)^m\Leftrightarrow cyc(\{1,\infty\})\subset Cl(\Phi(p))^m.\]
Applying the previous theorem twice (actually once in each direction) we get that:
\[cyc_\infty(p)\geqslant\frac 1m\Leftrightarrow cyc_\infty(\Phi(p))\geqslant\frac 1m.\]
This is enough to deduce the conclusion.
\end{proof}

\begin{te}\label{preservingclass1infty}
For an automorhpism $\Phi$ and $p\in cyc(\{1,\infty\})$ we have $Cl(p)=Cl(\Phi(p))$.
\end{te}
\begin{proof}
The conjugacy class of an element $p$ in $cyc_\infty(\{1,\infty\})$ is completely characterized by $cyc_\infty(p)$. Let $m,j\in\nz^*$ and $q\in  cyc_\infty(\{1,\infty\})$ be such that $cyc_\infty(q)=1/m$. By Proposition \ref{inftycycles} we have $p\in Cl(q)^j\Leftrightarrow cyc_\infty(p)\leqslant j/m$ and $\Phi(p)\in Cl(\Phi(q))^j\Leftrightarrow cyc_\infty(\Phi(p))\leqslant j\cdot cyc_\infty(\Phi(q))$ . We know that $p\in Cl(q)^j\Leftrightarrow \Phi(p)\in Cl(\Phi(q))^j$. By the previous corollary we get that $cyc_\infty(\Phi(q))<1/(m-1)$. In the end we have:
\[ cyc_\infty(p)\leqslant j/m\Rightarrow  cyc_\infty(\Phi(p))\leqslant j/(m-1) \ \ \ \forall m,j\in\nz^*.\]
Suppose that $cyc_\infty(p)< cyc_\infty(\Phi(p))$. Then there exists a suficiently large $m\in\nz$ and $j<m$ such that:
\[cyc_\infty(p)<\frac jm<\frac{j+1}m<cyc_\infty(\Phi(p)).\]
Because $j/(m-1)\leqslant (j+1)/m$ we get a contradiction. It follows that $cyc_\infty(p)\geqslant cyc_\infty(\Phi(p))$. Replacing $\Phi$ with $\Phi^{-1}$ we get the reverse inequality.
\end{proof}

\section{A theorem about product of two cycles}

The following theorem and its extension to ultraproducts are crucial for the proof of the main result.

\begin{te}[\cite{HKL} Theorem 7]\label{productofcycles}
Let $\sigma\in S_n$ and let $l_1$, $l_2$ be two natural numbers with $2\leqslant l_2\leqslant l_1\leqslant n$. 
Then there exists $C_1,C_2\in S_n$ cycles of lengths $l_1$ and $l_2$ such that $\sigma=C_1C_2$ if and only if $C_1C_2$ 
is the canonic decomposition of $\sigma$ in cycles or if the following holds:
\begin{enumerate}
\item $l_1+l_2=m(\sigma)+n(\sigma)+2s$, for some $s\in\nz$ and
\item $l_1-l_2\leqslant m(\sigma)-n(\sigma)$.
\end{enumerate}
\end{te}

For a permutation $p\in S_n$ denote by $dcd(p)$ the canonic decomposition of $p$ into cycles. For eample, if $dcd(\sigma)=(1,2,\ldots,m_1)(m_1+1,\ldots,m_2)\ldots(m_{n(\sigma)-1}+1,\ldots,m_{n(\sigma)})$ one can check that:
\[\sigma=(1,2\ldots,m_{n(\sigma)})(m_{n(\sigma)-1}+1,\ldots,m_2+1,m_1+1,1).\]
The two cycles have lengths $m(\sigma)$ and $n(\sigma)$ respectively. By some tricks in $S_n$ one can add one element to both cycles, or increase the smaller cycle by one while decreasing the larger cycle by one. The detailes can be checked in the cited article.

Now we need to extend this result in $\Pi_{k\to\omega}S_n$. We get the following theorem.

\begin{te}
Let $p\in\Pi_{k\to\omega}S_{n_k}$ and $q_1,q_2\in cyc(\{1,\infty\})$. Assume that $cyc_\infty(q_1)\geqslant cyc_\infty(q_2)>0$. Then $p\in Cl(q_1)Cl(q_2)$ if and only if:
\begin{enumerate}
\item $m(p)+n(p)\leqslant cyc_\infty(q_1)+cyc_\infty(q_2);$
\item $m(p)-n(p)\geqslant cyc_\infty(q_1)-cyc_\infty(q_2).$
\end{enumerate}
\end{te}
\begin{proof}
Let $p,q_1,q_2$ as in the hypothesis be such that $p=q_1q_2$. Let $q_i=\Pi_{k\to\omega}q_i^k$, $i=1,2$, where each $q_i^k\in S_{n_k}$ is composed of cycle of length $l_i^k$.
Then $\lim l_i^k/n_k=cyc_\infty(q_i)$.

Let $p^k=q_1^kq_2^k$ such that $p=\Pi_{k\to\omega}p^k$. If $\{i:q_1^kq_2^k=dcd(p^k)\}\in\omega$ then $p\in cyc(\{1,\infty\})$, $m(p)=cyc_\infty(p)=cyc_\infty(q_1)+cyc_\infty(q_2)$
and $n(p)=0$. The second condition is trivial.

Thus, we may assume that $q_1^kq_2^k\neq dcd(p^k)$ for all $k$. Then by the previous theorem, $l_1^k+l_2^k\geqslant m(p^k)+n(p^k)$ and $l_1^k-l_2^k\leqslant m(p^k)-n(p^k)$.
Dividing by $n_k$ and passing to limit we get the conclusion.

Now let $p,q_1,q_2$ be as in the hypothesis such that conditions $(1)$ and $(2)$ hold. Assume that $n(p)>0$. Let $p=\Pi_{k\to\omega}p^k$. We want to change permutations $p^k$ to get the following inequalities:
\[m(p)-\frac{m(p^k)}{n_k}-\frac2{n_k}\geqslant 0;\ \ \ n(p)-\frac{n(p^k)}{n_k}-m(p)+\frac{m(p^k)}{n_k}-\frac3{n_k}\geqslant 0.\]
For the first inequality we increase the number of fixed points of $p^k$ until we get it. For the second inequality we glue some cycles in $p^k$ leaving $m(p^k)$ unchanged. We can make a sufficient small number of changes such that the element  $p=\Pi_{k\to\omega}p^k$ is unchanged. Define now:
\begin{align*}
l_1^k=&m(p^k)+[n_k(cyc_\infty(q_1)-m(p))]+c;\\
l_2^k=&n(p^k)+[n_k(cyc_\infty(q_2)-n(p))]+2.
\end{align*}
One requirement of Theorem \ref{productofcycles} is that $l_1^k+l_2^k-m(p^k)-n(p^k)$ is even. We choose $c\in\{0,1\}$ in order to satisfy this requirement. Now let's check the inequalities. Clearly $l_2^k\geqslant 2$ and $l_1^k+l_2^k\geqslant m(p^k)+n(p^k)$. Now:
\begin{align*}
l_1^k-l_2^k\geqslant& m(p^k)+n_k(cyc_\infty(q_1)-m(p))-n(p^k)-n_k(cyc_\infty(q_2)-n(p))+c-3\\
\geqslant&n_kn(p)-n(p^k)-n_km(p)+m(p^k)-3\geqslant 0.
\end{align*}
Also:
\begin{align*}
l_1^k-l_2^k\leqslant& m(p^k)+n_k(cyc_\infty(q_1)-m(p))-n(p^k)-n_k(cyc_\infty(q_2)-n(p))+c-1\\
\leqslant&  m(p^k)-n(p^k)-n_k(m(p)-n(p)-cyc_\infty(q_1)+cyc_\infty(q_2))\leqslant m(p^k)-n(p^k).
\end{align*}
Lastly:
\[l_1^k\leqslant m(p^k)-n_km(p)+1+c+n_kcyc_\infty(q_1)\leqslant n_k.\]
From Theorem \ref{productofcycles} we deduce the existence of $q_1^k,q_2^k$ cycles of length $l_1^k$ and $l_2^k$ respectively such that $p^k=q_1^kq_2^k$. By construction $\lim_{k\to\omega}l_i^k/n_k=cyc_\infty(q_i)$ so $\Pi_{k\to\omega}q_i^k\in Cl(q_i)$ for $i=1,2$. It follows that $p\in Cl(q_1)Cl(q_2)$.

If $n(p)=0$ then $p\in cyc(\{1,\infty\})$ and $m(p)=cyc_\infty(p)$. We have the inequalities  $cyc_\infty(q_1)-cyc_\infty(q_2)\leqslant cyc_\infty(p)\leqslant  cyc_\infty(q_1)+cyc_\infty(q_2)$. There are two cases, whenever $cyc_\infty(p)>cyc_\infty(q_1)$ or $cyc(p)\leqslant cyc_\infty(q_1)$. These cases can be setteled as in the proof of Proposition \ref{inftycycles}.
\end{proof}

\section{The main result}

\begin{te}\label{tracepreserving}
Each automorphism of  $\Pi_{k\to\omega}S_{n_k}$ preserves the Hamming distance.
\end{te}
\begin{proof}
Let $p\in\Pi_{k\to\omega}S_{n_k}$ and $q_1,q_2\in cyc(\{1,\infty\})$ such that $cyc_\infty(q_1)=m(p)$ and $cyc_\infty(q_2)=n(p)$. By the previous theorem $p\in Cl(q_1)Cl(q_2)$. From Theorem \ref{preservingclass1infty} we deduce that $Cl(q_i)=Cl(\Phi(q_i))$ so  $\Phi(p)\in Cl(q_1)Cl(q_2)$. Using again the previous theorem in the opposite direction we get the inequalities:
\[m(\Phi(p))+n(\Phi(p))\leqslant m(p)+n(p)\mbox{ and } m(\Phi(p))-n(\Phi(p))\geqslant m(p)-n(p).\]
From these inequalities we deduce $n(\Phi(p))\leqslant n(p)$. We can get the reverse inequality using $\Phi^{-1}$, so  $n(\Phi(p))=n(p)$. Now we also get  $m(\Phi(p))=m(p)$.

This is equivalent to $cyc_1(\Phi(p))=cyc_1(p)$, so the Hamming distance between an arbitrary element and the identity is preserved. This is enough because the Hamming distance is bi-invariant.
\end{proof}

\begin{te}
Any automorphism of $\Pi_{k\to\omega}S_{n_k}$ is class preserving.
\end{te}
\begin{proof}
Let $\Phi$ be an automorphism of $\Pi_{k\to\omega}S_{n_k}$ and let $p\in\Pi_{k\to\omega}S_{n_k}$. By Theorem \ref{tracepreserving} we get that $cyc_1(p)=cyc_1(\Phi(p))$. It is easy to see that $cyc_1(p^i)=\sum_{j|i}cyc_j(p)$ for any $i\in\nz^*$. Now the equality $cyc_i(p)=cyc_i(\Phi(p))$ follows from Theorem \ref{tracepreserving} and an induction step.

The result follows now from Proposition \ref{conjugacyclasses}.
\end{proof}

LIVIU P\u AUNESCU, \emph{INSTITUTE of MATHEMATICS "S. Stoilow" of the ROMANIAN ACADEMY} email:
liviu.paunescu@imar.ro
\end{document}